\documentclass[a4paper,12pt]{article}
\usepackage{authblk}
\usepackage{a4wide}
\usepackage{amsfonts}
\usepackage{epsfig}
\usepackage{makeidx}
\usepackage{amsmath,amssymb,latexsym}
\usepackage{enumerate}
\usepackage{graphicx}

\newcommand{\VO}{$B_0$-VPG}
\newcommand{\F}{{\mathcal F}}

\newtheorem{thm}{Theorem}
\newtheorem{defi}[thm]{Definition}
\newtheorem{prop}[thm]{Proposition}
\newtheorem{lem}[thm]{Lemma}
\newtheorem{cor}[thm]{Corollary}

\newtheorem{remark}{Remark}

\newenvironment{proof}{\textit{Proof.}}{\hfill $\Box$ \\}

\newcommand\blfootnote[1]{%
  \begingroup
 \renewcommand\thefootnote{}\footnote{#1}%
  \addtocounter{footnote}{1}%
  \endgroup
}

\begin{document}

\title{Vertex intersection graphs of paths on a grid: characterization within block graphs}

\author[a]{Liliana Alc\'on}
\author[b,c]{Flavia Bonomo}
\author[a,c]{Mar\'\i a P\'\i a Mazzoleni}

\affil[a]{Departamento de Matem\'atica, FCE-UNLP, La Plata,
Argentina.}
\affil[b]{Departamento de Computaci\'on, FCEN-UBA, Buenos Aires,
Argentina.}
\affil[c]{CONICET}

\date{}

\maketitle

\vspace{-1cm}

\blfootnote{E-mail addresses: liliana@mate.unlp.edu.ar;
fbonomo@dc.uba.ar;
pia@mate.unlp.edu.ar. \\
This work was partially supported by UBACyT Grant
20020130100808BA, CONICET PIP 122-01001-00310 and
112-201201-00450CO and ANPCyT PICT 2012-1324 (Argentina).}

\begin{abstract}
We investigate graphs that can be represented as vertex
intersections of horizontal and vertical paths in a grid, the so
called $B_0$-VPG graphs. Recognizing this class is an NP-complete
problem. Although, there exists a polynomial time algorithm for
recognizing chordal $B_0$-VPG graphs. In this paper, we present a
minimal forbidden induced subgraph characterization of $B_0$-VPG
graphs restricted to block graphs.  As a byproduct, the proof of the main
theorem provides an alternative certifying recognition and representation
algorithm for $B_0$-VPG graphs in the class of block graphs.

\

\end{abstract}

\section{Introduction}

A \textit{VPG representation} of a graph $G$ is a collection of
paths of the two-dimensional grid where the paths represent the
vertices of $G$ in such a way that two vertices of $G$ are
adjacent in $G$ if and only if the corresponding paths share at
least one vertex of the grid. A graph which has a VPG
representation is called a \textit{VPG graph}. In this paper, we
consider the subclass $B_0$-VPG.\

A \textit{B$_0$-VPG representation} of $G$ is a VPG representation
in which each path in the representation is either a horizontal
path or a vertical path on the grid. A graph is a
\textit{B$_0$-VPG graph} if it has a $B_0$-VPG representation.\

Representations by  intersections of paths on grids arise naturally in
the context of circuit layout problems and layout optimization
\cite{sinden} where a layout is modelled as paths (wires) on a
grid. Often one seeks to minimize the number of times a wire is
bent \cite{bandy,molitor} in order to minimize the cost or
difficulty of production. Other times layout may consist of
several layers where the wires on each layer are not allowed to
intersect. This is naturally modelled as the coloring problem on
the corresponding intersection graph.\

The recognition problem is NP-complete for both VPG and $B_0$-VPG
graphs (see \cite{asinowski} for more details about this and
related results). Since all interval graphs are $B_0$-VPG graphs,
it is natural to consider other subclasses of chordal $B_0$-VPG
graphs. In \cite{golumbic}, certain subclasses of $B_0$-VPG graphs
have been characterized and shown to admit a polynomial time
recognition; namely split, chordal claw-free and chordal bull-free
$B_0$-VPG graphs. Recently, in \cite{chaplick} the authors present
a polynomial time algorithm for deciding whether a given chordal
graph is a $B_0$-VPG graph. In \cite{asinowski2}, it was shown
that chordal $B_0$-VPG graphs are equivalent to the strongly
chordal $B_0$-VPG graphs.\

 In this paper, we consider $B_0$-VPG graphs more from a structural point of view.
We present a minimal forbidden induced subgraph characterization
of $B_0$-VPG
graphs restricted to block graphs. As a byproduct, the proof of the main
theorem provides an alternative recognition and representation
algorithm for $B_0$-VPG graphs in the class of block graphs.

\section{Preliminaries}
\label{s:preliminares} In this paper all graphs are connected,
finite and simple. Notation we use is that used by Bondy and Murty
\cite{bondy}.\

 Let $G=(V,E)$ be a graph with vertex set $V$ and
edge set $E$.

%For a vertex $v\in V$, we let $N(v)$ denote the set of vertices in
%$G$ that are adjacent to $v$, that is, the neighbors of $v$.
%$N(v)$ is called the \textit{neighborhood} of vertex $v$. We will
%write $N[v]=N(v)\cup \{v\}$, and call $N[v]$ the \textit{closed
%neighborhood} of vertex $v$. For $X\subseteq V$ we define
%$N(X)=\bigcup_{v\in X} N(v)$ and $N[X]= \bigcup_{v\in X} N[v]$. We
%denote by $G[X]$ the subgraph induced by $X$. A vertex is called
%\textit{simplicial} if $G[N(v)]$ is a clique.\

We write $G-v$ for the subgraph obtained by deleting a vertex $v$
and all the edges incident to $v$. Similarly, for $A\subseteq V$,
we denote by $G-A$ the subgraph of $G$ obtained by deleting the
vertices in $A$ and all the edges incident to them, that is,
$G-A=G[V\backslash A]$. \

If $H$ is a graph, a graph $G$ is \emph{$H$-free} if $G$ contains
no induced $H$ subgraph isomorphic to $H$. If $\mathcal H$ is a
collection of graphs, the graph $G$ is \emph{$\mathcal H$-free} if
$G$ is $H$-free for every $H\in\mathcal H$. \

%We denote by $P_n$, $n\geq 2$, the induced path on $n$ vertices.
%We denote by $C_n$, $n\geq 3$, the chordless cycle on $n$ vertices.

A \textit{complete} is a set of pairwise adjacent vertices.
% and a \textit{stable set} is a set of pairwise nonadjacent vertices.
A \textit{clique} is a complete which is not properly contained in
another complete. A \textit{thin spider} $N_n$ is the graph whose
$2n$ vertices can be partitioned into a clique $K=\{c_1,...c_n\}$
and a set $S=\{s_1,...,s_n\}$ of pairwise nonadjacent vertices
such that, for $1\leq i,j\leq n$, $s_i$ is adjacent to $c_j$ if
and only if $i=j$. We say that $N_n$ is a thin spider of size $n$.

%A \textit{VPG representation} of a graph $G$ is a collection of
%paths of the two-dimensional grid where the paths represent the
%vertices of $G$ in such a way that two vertices of $G$ are
%adjacent in $G$ if and only if the corresponding paths share at
%least one vertex of the grid. A graph which has a VPG
%representation is called a \textit{VPG graph}. In this paper, we
%consider the subclass $B_0$-VPG.\

%A \textit{B$_0$-VPG representation} of $G$ is a VPG representation
%in which each path in the representation is either a horizontal
%path or a vertical path on the grid. A graph is a
%\textit{B$_0$-VPG graph} if it has a $B_0$-VPG representation.\

The following lemma will be use in our paper.

\begin{lem} \cite{asinowski2} In a $B_0$-VPG representation of a clique, all the
corresponding paths share a common grid point.\end{lem}

We will distinguish between two types of $B_0$-VPG representations
of a clique: a \textit{line clique} and a \textit{cross clique}.
We say that a clique is represented as a line clique if all paths
corresponding to the vertices of the clique use a common row or a
common column and intersect on at least one grid point of that row
or column. A clique is said to be represented as a cross clique if
the paths corresponding to the vertices of the clique share
exactly one grid point, say $(x_i,y_j)$, and there exists at least
one such path which uses column $x_i$ and at least one such path
which uses row $y_j$. The grid point $(x_i,y_j)$ is called the
\textit{center} of the cross clique (see Figure \ref{f:cliques}
for examples). It is easy to see that any $B_0$-VPG representation
of a clique is either a line clique or a cross clique.\

\begin{figure}[h]
\centering{
\includegraphics[width=.7\textwidth]{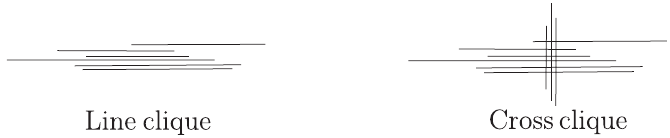}
\caption{A line clique and a cross clique.}\label{f:cliques}}
\end{figure}

\

%The following lemma will be used in several occasions throughout
%this paper. A \textit{diamond} is a graph obtained from $K_4$ by
%deleting exactly one edge.

%\begin{lem}\cite{golumbic} Let $G$ be a diamond with vertex set
%$V=\{a,b,c,d\}$ and edge set $E=\{ab,ac,bc,bd,cd\}$. Then in any
%%$B_0$-VPG representation of $G$, $P_b$ and $P_c$ use a common
%horizontal or a common vertical grid line.\end{lem}

%The following lemma concerning the neighborhood of any vertex in a
%$B_0$-VPG graph $G$ will be also used often.

%\begin{lem}\cite{golumbic} Let $G=(V,E)$ be a $B_0$-VPG graph.
%Then $G[N[v]]$ is an interval graph for every vertex $v$ in
%$G$.\end{lem}

\section{Block graphs} \label{s:our results}

In this Section we will give a characterization of $B_0$-VPG
graphs restricted to block graphs by a family of minimal forbidden induced subgraphs.

%We will distinguish between graphs which have a diamond as induced
%subgraph and graphs which have not it.

%We have to remember the following definition:

\begin{defi} A \textit{block graph} is a connected graph in which every two-connected component (block) is a clique.
\end{defi}

A \textit{diamond} is a graph obtained from $K_4$ by deleting
exactly one edge. A graph is called \textit{chordal} if it does
not contain any chordless cycle of length at least four. It is
known that block graphs are connected chordal diamond-free graphs.

A \textit{cutpoint} is a vertex whose removal from the graph increases the number of connected components.
% That is, it makes some points unreachable from some others. It disconnects the graph.

\begin{defi} Let $G$ be a block graph. An \textit{endblock} is a block
having exactly one cutpoint. An \textit{almost endblock} is a
block $B$ having at least two cutpoints and such that exactly one of these cutpoints belongs to blocks
(different from $B$) that are not endblocks. An \textit{internal
block} is a block that is neither an endblock nor an almost
endblock.

We will call \textit{3-cutpoints} to cutpoints that belong to
exactly 3 blocks, and \textit{2-cutpoints} to cutpoints that
belong to exactly 2 blocks, one of which is an endblock.\end{defi}

\begin{defi} \cite{harari-prins} The \textit{block-cutpoint-tree} $bc(G)$ of a graph $G$ is a graph whose vertices are in one-to-one correspondence with the blocks and cutpoints of $G$, and such that two vertices of $bc(G)$
are adjacent if and only if one corresponds to a block $H$ of $G$
and the other to a cutpoint $c$ of $G$, and $c$ is in $H$.
\end{defi}

%\begin{defi}
%A \textit{pruning sequence} of a tree is an ordering $\{v_1,
%\dots, v_n\}$ of its vertices, such that each $v_i$ is a leaf of
%the subgraph (indeed subtree) induced by the vertices $\{v_1,
%\dots v_i\}$.
%\end{defi}

The graph $N_5$, defined in \cite{golumbic}, is the thin spider of
size $5$, i.e., is a split graph which consists of a clique graph
$\{c_1,\dots, c_5\}$, and a set $\{s_1,\dots, s_5\}$ of pairwise nonadjacent vertices such
that $s_i$ is adjacent to $c_j$ if and only if $i = j$.

We let $\F$ denote the family of block graphs obtained from $N_5$
by a finite sequence of applications of the following procedure:
let H be a complete subgraph of size 4 in $G$ having at least two
2-cutpoints, say $v_1$ and $v_2$, with endblocks $B_1$ and $B_2$,
respectively. We contract $v_1$ and $v_2$ into a single vertex
$x$. Then, we replace $B_1 -\{x\}$ and $B_2 - \{x\}$ by two thin
spiders of size 3, making $x$ adjacent to the vertices of the
cliques of both the spiders. In Figure \ref{f:BLOCK} we offer some
examples of graphs in $\F$.

\begin{figure}[h]
\centering{
\includegraphics[width=\textwidth]{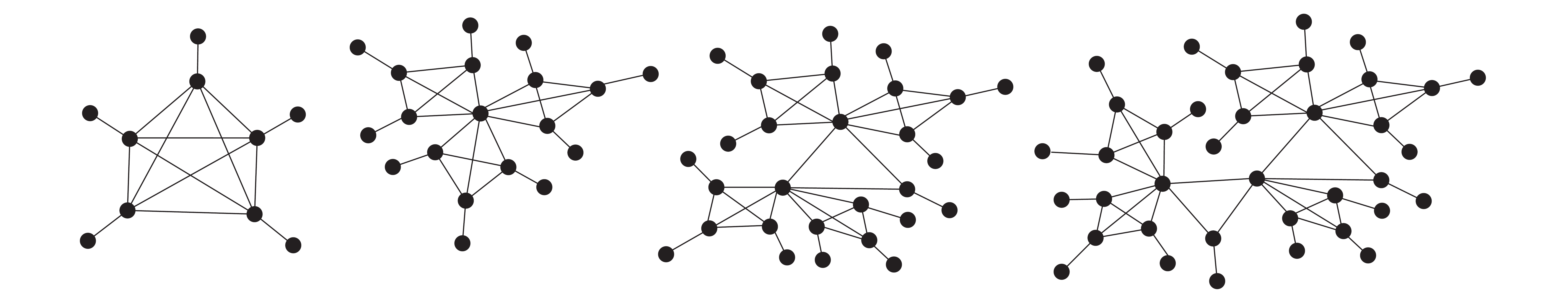}
\caption{Some examples of graphs in  $\F$.}\label{f:BLOCK}}
\end{figure}

\begin{prop}\label{prop:F}
Properties of graphs in $\F$, different from $N_5$:
\begin{enumerate}[i.]
\item\label{i:1} each block is of size at most 4;

\item\label{i:2} all the vertices are either leaves, 2-cutpoints
or 3-cutpoints;

\item\label{i:3} the endblocks are of size 2 and have a
2-cutpoint;

\item\label{i:4} the almost endblocks are of size 4 and have three
2-cutpoints and one 3-cutpoint;

\item\label{i:5} the internal blocks are of size 3 and have one
2-cutpoint and two 3-cutpoints;

\item\label{i:6} a graph in $\F$ obtained from $N_5$ by applying
the procedure $k$ times, $k \geq 1$, has $6(k+1)$ blocks
($4(k+1)+1$ endblocks, $k+2$ almost endblocks, and $k-1$ internal
blocks), $5(k+1)$ cutpoints ($k$ 3-cutpoints and $4(k+1)+1$
2-cutpoints), and $9(k+1)+1$ vertices.
\end{enumerate}
\end{prop}

\begin{proof}
We will prove it by induction on the number of times we apply the
procedure. By symmetry of $N_5$, there is only one graph obtained
by applying the procedure once (Figure \ref{f:BLOCK}), and it has
no internal blocks. It is easy to verify that this graph satisfies
the properties claimed.

Suppose the properties are satisfied by all graphs in $\F$
obtained from $N_5$ by applying the procedure $k$ times, $k \geq
1$, and let $G$ be one such graph. Let us apply the procedure once
more. Let $H$ be a complete subgraph of size 4 in $G$. By
inductive hypothesis, $H$ is an almost endblock of $G$, and has
three 2-cutpoints and one 3-cutpoint. By item \emph{\ref{i:3}},
the blocks incident to the 3-cutpoint are not endblocks.

Choose two vertices $v_1$ and $v_2$ which are 2-cutpoints, and let
$B_1$ and $B_2$ be the endblocks incident with $v_1$ and $v_2$,
respectively. By item \emph{\ref{i:3}}, $B_1$ and $B_2$ are of
size 2. Contract $v_1$ and $v_2$ into a single vertex $x$, and
replace $B_1 - \{x\}$ and $B_2 - \{x\}$ by two thin spiders of
size 3, induced respectively by the vertices
$\{c_1,c_2,c_3,s_1,s_2,s_3\}$ and
$\{c'_1,c'_2,c'_3,s'_1,s'_2,s'_3\}$, making $x$ adjacent to the
vertices of the cliques of both the spiders, i.e,
$\{c_1,c_2,c_3,c'_1,c'_2,c'_3\}$.

After the procedure, $H' = H -\{v_1,v_2\} \cup \{x\}$ is a block
of size 3, and it has two 3-cutpoints and still one 2-cutpoint.
The new blocks $\{c_1,c_2,c_3,x\}$ and $\{c'_1,c'_2,c'_3,x\}$ are
almost endblocks, they are of size 4 and have three 2-cutpoints
and one 3-cutpoint, namely $x$. And since the blocks incident to
the other 3-cutpoint of $H'$ are not endblocks, $H'$ is an
internal block. The six new endblocks $\{c_i,s_i\}$ and
$\{c'_i,s'_i\}$, $i=1,2,3$ have a 2-cutpoint each (vertices $c_i$
and $c_i'$) and a leaf each (vertices $s_i$ and $s_i'$). The
remaining blocks as well as their conditions are not affected. So
items \emph{\ref{i:1}--\ref{i:5}} are satisfied by the new graph.
To see item \emph{\ref{i:6}}, notice that we have replaced 2
endblocks by 8 new blocks, 6 of which are endblocks and 2 of which
are almost endblocks. Also, one almost endblock has become an
internal block. We have replaced 4 vertices by 13 vertices and, in
particular, two 2-cutpoints by one 3-cutpoint and six 2-cutpoints.
\end{proof}

\begin{cor} The family $\F$ is infinite. \end{cor}

\begin{proof}
By Proposition \ref{prop:F}, for every graph in $\F$ there is
always an almost endblock on which we can perform the procedure in
order to obtain a new graph in $\F$ with strictly more vertices.
\end{proof}

\begin{cor}\label{cor:minimal} Each graph in $\F$ is minimal, i.e., it does not contain another graph in $\F$ as induced subgraph. \end{cor}

\begin{proof}
Let $G \in \F$ and let $G'$ be a proper connected induced subgraph
of $G$. The blocks of $G'$ are the blocks of $G$ intersected with
$V(G')$. Suppose $G' \in \F$, and suppose $B'$ is a block of $G'$
such that $B' = B \cap V(G')$, with $B$ a block of $G$, and $|B'|
< |B|$. Then, $B$ cannot be an endblock of $G$ because, by
Proposition \ref{prop:F}.\emph{\ref{i:3}}, endblocks of $G$ have
size 2 and $|B'| < |B|$; $B'$ cannot be an almost endblock of $G'$
because by Proposition \ref{prop:F}.\emph{\ref{i:1}} $B$ has at
most 4 vertices, and by item \emph{\ref{i:4}} $B'$ should have 4
vertices; $B'$ cannot be an internal block of $G'$ because, in
that case, by Proposition \ref{prop:F} and the cardinalities of
each type of block, $B$ should be an almost endblock but, by item
\emph{\ref{i:5}}, $B'$ should have two 3-cutpoints while $B$ has
only one 3-cutpoint, and no 2-cutpoint of $G$ may become a
3-cutpoint in an induced subgraph of it. So, $B'$ is an endblock
and $B$ is either an almost endblock or an internal block. Let $x$
be the cutpoint of $B'$ in $G'$. By Proposition
\ref{prop:F}.\emph{\ref{i:3}}, $x$ is a 2-cutpoint of $G'$. If $x$
is a 2-cutpoint in $G$, as $B$ is not an endblock, we have that
$G' = P_3$, and it does not belong to $\F$ (by Proposition
\ref{prop:F}.\emph{\ref{i:6}}). If $x$ is a 3-cutpoint in $G$, let
$B_1$ and $B_2$ be the other two blocks in $G$ that contain $x$.
Since $x$ is a 2-cutpoint in $G'$, the intersection of one of
these blocks with $V(G')$ is $\{x\}$. Without loss of generality,
suppose this is the case of $B_2$. If $B_1$ is an almost endblock
in $G$, then $G'$ is an induced subgraph of the thin spider $N_4$,
that is not in $\F$ (by Proposition
\ref{prop:F}.\emph{\ref{i:6}}). If $B_1$ is an internal block, by
cardinality, it may be either an endblock or an internal block in
$G'$. In the first case, $G' = P_3$, that is not in $\F$. The
second case cannot arise, because $B_1$ cannot have two
3-cutpoints in $G'$ (no 2-cutpoint of $G$ may become a 3-cutpoint
in an induced subgraph of it).
\end{proof}

We will prove now some properties about the $B_0$-VPG
representations of block graphs.

\begin{lem}\label{l:cross} If a clique $K$ of a block graph $G$ has 3 cutpoints, then,
in a $B_0$-VPG representation of $G$, it has to be represented as
a cross clique. Similarly, if the clique $K$ has 4
cutpoints.\end{lem}

\begin{proof} Let $v_i$, $1\leq i\leq 3$, be the cutpoints of $K$. Since $v_i$, $1\leq i\leq 3$, are cutpoints there exist
vertices $x_j$, $1\leq j\leq 3$, such that $v_i$ is adjacent to
$x_j$ if and only if $i=j$. Suppose that the clique $K$ is
represented as a line clique. So, all the paths which represent
vertices of $K$ are horizontal (respectively vertical) paths using
a common row (respectively column) of the grid. Suppose that
$P_{v_1}$ is the farthest line in the East direction (by
\textit{farthest line} in some direction, in the context of a
clique whose paths intersect at point $p$ of the grid, we mean the
path belonging to the clique and such that one of its endpoints
maximizes the distance to $p$ in that direction) and $P_{v_2}$ is
the farthest lines using the West. But, $P_{v_3}$ is an horizontal
(respectively vertical) path lying in the same row (respectively
column) that $P_{v_1}$ and $P_{v_2}$ and it has to be adjacent to
$P_{x_3}$, and $P_{x_3}$ is not adjacent with $P_{v_1}$ and
$P_{v_3}$. So, it is impossible to represent $P_{x_3}$.\

In a similar way, it is easy to see that the results follows if
$K$ has 4 cutpoints.
\end{proof}

\begin{lem}\label{lem:cardinals}
If a clique $K$ of a block graph $G$ has 4 cutpoints, then, then,
in a $B_0$-VPG representation of $G$, the 4 cutpoints are
represented as the farthest lines South, North, West and East,
respectively. Similarly, if a clique $K$ has 3 cutpoints, then
they are represented as the farthest lines of three different
cardinal points.
\end{lem}

\begin{proof}
Suppose that $K$ has 4 cutpoints. By Lemma \ref{l:cross}, $K$ has
to be represented as a cross clique. Using the same idea that in
the proof of Lemma \ref{l:cross}, it is easy to see that the 4
cutpoints are represented as the farthest lines South, North, West
and East, respectively.

In a similar way, it is easy to see that the results follows if
$K$ has 3 cutpoints.
\end{proof}

\begin{lem}\label{l:orden}
In any $B_0$-VPG representation of the graph $W$, given in Figure
\ref{f:W}, the intersection points of cliques $C_1$, $K$, and
$C_2$ lie in a same line of the grid, and the intersection point
of clique $K$ lies between the intersection points of cliques
$C_1$ and $C_2$.
\end{lem}

\begin{proof}
Let $x_1$, $x_2$, $x_3$ be the intersection points in the grid of
the cliques $C_1$, $K$, $C_2$,  respectively. Suppose, by the
contrary, that there is a $B_0$-VPG representation of the graph
$W$ such that $x_2$ does not lie between $x_1$ and $x_3$. Without
loose of generality, we can assume that $x_2$ is to the left of
$x_1$ and $x_3$. By Lemmas \ref{lem:cardinals} and \ref{l:cross},
since $C_1$ and $C_2$ have 4 cutpoints, they are represented as
cross cliques where the 4 cutpoints are the farthest lines South,
North, West and East, respectively. So, it is impossible to
represent the vertex $w$ of $W$.
\end{proof}

\begin{remark}\label{r:1} Observe that all the graphs of $\F - \{N_5\}$ have $W$ as
induced subgraph.\end{remark}

\begin{figure}[h]
\centering{
\includegraphics[height=1.5in]{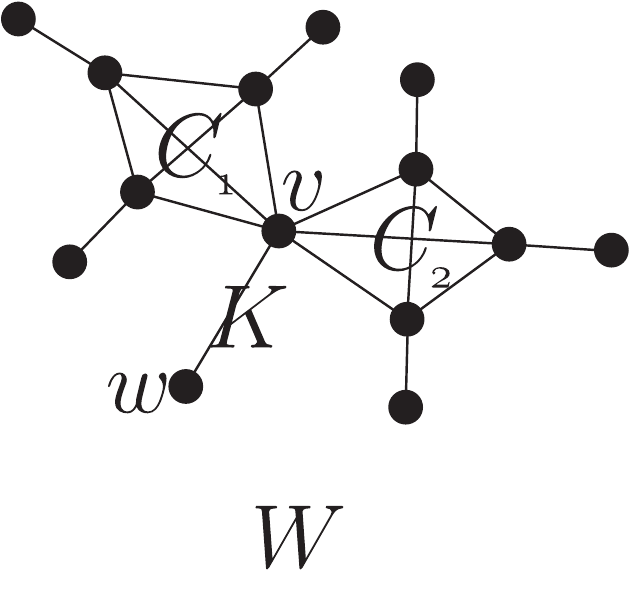}
\caption{The graph $W$.}\label{f:W}}
\end{figure}

\begin{lem}\label{lem:F-not-V}
The graphs of $\F$ are not \VO.
\end{lem}

\begin{proof}
The graph $N_5$ is not \VO  \ \cite{golumbic}. We will proceed by
induction on the number of applications of the procedure in the
construction of the graph from $N_5$. Assume that if we applied
the procedure $k$ times, then we obtain a graph of $\F$ which is
not \VO.

Let $G$ be a graph of $\F$ which is obtained applying the
procedure $k+1$ times. Suppose, on the contrary, that $G\in
B_0$-VPG. We take a $B_0$-VPG representation of $G$.

By Remark \ref{r:1}, $G$ has $W$ as induced subgraph. Let $v$ be
the vertex of $W$ as in Figure \ref{f:W}, let $P_v$ be the path
which represents $v$ in the $B_0$-VPG representation of $G$ that
we took. Let $x_1$, $x_2$, $x_3$ be the intersection points  in
the grid of the cliques $C_1$, $K$, $C_2$, respectively. Clearly,
the three vertices lie in a same line of the grid and, by Lemma
\ref{l:orden}, $x_2$ lies between $x_1$ and $x_3$.

We are going to construct a new $B_0$-VPG representation. This is
obtained of the previous one by removing the paths which
correspond to $C_1$, $C_2$ and their corresponding endblocks; and
adding the paths $P_{v_i}$, with $1\leq i\leq 4$, such that
$V(P_{v_1})=\{x_1,x_2\}$, $V(P_{v_2})=\{x_2,x_3\}$,
$V(P_{v_3})=\{x_1\}$ and $V(P_{v_4})=\{x_3\}$. Observe that this
is a $B_0$-VPG representation of a graph of $\F$ that was obtained
applying the procedure $k$ times, which is a contradiction.

Hence, the graphs of $\F$ are not \VO.\end{proof}

We have proved the following theorem:

\begin{thm}\label{t:main}
Let $G$ be a block $VPG$  graph. Then $G$ is \VO \ if and only if
$G$ is $\F$-free. Moreover, the graphs of $\F$ are minimal not
\VO.
\end{thm}

\begin{proof}
The only if part follows from Lemma \ref{lem:F-not-V}. For the if
part, let $G$ be a block $\F$-free graph. Let $s$ be a BFS
ordering of the vertices of the block-cutpoint-tree $bc(G)$, in
such a way that $s_1$ is a vertex of $bc(G)$ corresponding to a
block of $G$. Let $H_i$ be the $i$-th block in $s$. We will
consider the graph $G_i$ as the graph induced by the first $i$
blocks $H_1,\dots, H_i$ in $s$, and proceed by induction on $i$.
Notice that the graph $G_i$ is connected and that $H_i$ is an
endblock of $G_i$; moreover, by the BFS algorithm, if $i>1$, there
is only one cutpoint $c$ of $G$ belonging to $H_i$ and appearing
in $s$ before $H_i$. We will denote that cutpoint as $c(i)$.
Notice that $c(i)$ is a cutpoint of $G_i$. All the blocks between
$c(i)$ and $H_i$ containing $c(i)$ are endblocks of $G_i$ and are
consecutive in $s$. For each such block $H_j$, it holds
$c(j)=c(i)$. \

For each cutpoint $c$ of $G$, there is only one block containing
$c$ and appearing before $c$ in $s$. We will denote that block by
$H^c$.

We will label the cutpoints of $G$ as $A$ or $B$, according to
some rules, in decreasing order with respect to $s$. As $s$ was
obtained by a BFS of $bc(G)$, by the moment of labeling the
cutpoint $c$, all the other cutpoints of the blocks containing $c$
and different from $H^c$ are already labeled. The cutpoint $c$
will be labeled $B$ if it belongs to at least two blocks,
different from $H^c$, such that each of them either has at least
four cutpoints or has exactly three cutpoints and one of them is
already labeled $B$. The cutpoint $c$ will be labeled $A$
otherwise.

We will show by induction on $i$ that we can find a \VO \
representation of $G_i$ such that if $c$ is a cutpoint of $G$ that
is a vertex of $G_i$, then it corresponds to the farthest North,
South, East or West line of the line or cross representation of
the clique $H^c$ and, moreover, if $c$ is labeled $B$, then it
corresponds to the farthest North and South, or East and West
(simultaneously) line of the line or cross representation of the
clique $H^c$.

\medskip

\textbf{Claim.} Since $G$ is $\F$-free, the following conditions
hold: $(i)$ no block of $G$ has five (or more) cutpoints; $(ii)$ a
cutpoint $c$ labeled $B$ belongs to exactly two blocks, different
from $H^c$, such that each of them either has at least four
cutpoints or has exactly three cutpoints and one of them
(different from $c$) is labeled $B$; $(iii)$ if a cutpoint $c$ is
labeled $B$, then $H^c$ has at most three cutpoints; and $(iv)$ no
block of $G$ having at least three cutpoints is $H^{c_1}$ and
$H^{c_2}$ for two cutpoints $c_1$ and $c_2$ labeled $B$.

\textit{Proof of the claim.} Condition $(i)$ holds since $G$ is
$N_5$-free. Let us assume from now on that $(i)$ is satisfied.

Suppose by contradiction that one of conditions $(ii)$, $(iii)$ or
$(iv)$ does not hold. We will prove, by induction in the number of
cutpoints labeled $B$ on $bc(G)$, that $G$ contains a member of
$\F$ as an induced subgraph.

If there is only one vertex $v$ labeled $B$, then the conditions
that should fail are $(ii)$ or $(iii)$. By the labeling rules and
since $v$ is the only vertex labeled $B$, it belongs to at least
two blocks, different from $H^v$, such that each of them has four
cutpoints. Either if the number of such blocks is at least three
or if $H^v$ has four cutpoints, then $G$ contains the second graph
in Figure \ref{f:BLOCK} as induced subgraph.

Suppose that the number of vertices labeled $B$ is greater than
one, and let $v$ be the first vertex labeled $B$ in the BFS
sequence $s$ (i.e., the one with higher index in $s$).

By the labeling rules and since $v$ is the first vertex labeled
$B$, it belongs to at least two blocks, different from $H^v$, such
that each of them has four cutpoints. Either if the number of such
blocks is at least three or if $H^v$ has four cutpoints, $G$
contains the second graph in Figure \ref{f:BLOCK} as induced
subgraph. Assume then that $v$ belongs to exactly two blocks,
different from $H^v$, such that each of them has four cutpoints,
and that $H^v$ has at most three cutpoints.

If $H^v$ has three cutpoints, let $w$ be other cutpoint of $G$
such that $H^v = H^w$. If $w$ is labeled $B$, since $s$ is a BFS
order and $v$ is the first vertex labeled $B$, $w$ belongs to at
least two blocks, different from $H^v$, such that each of them has
four cutpoints. Then $G$ contains the third graph in Figure
\ref{f:BLOCK} as induced subgraph. If $w$ is labeled $A$,
conditions $(ii)$, $(iii)$ and $(iv)$ are ``locally'' satisfied by
$v$, $w$, and $H^v$. We can replace $v$ and all the connected
components of $G - v$, except the one containing $H^v-v$, by four
vertices $v_1$, $v_2$, $v_1'$ and $v_2'$ by making $v_1$ and $v_2$
adjacent to each other and to $H^v-v$, $v_1'$ adjacent just to
$v_1$, and $v_2'$ adjacent just to $v_2$. Call $G'$ the obtained
graph. Now the block $H' = H^v-v \cup \{v_1,v_2\}$ of $G'$ has
four cutpoints (all of them labeled $A$), so the label of every
cutpoint placed before $v$ in $s$ remains unchanged in a labeling
of $bc(G')$, and the condition among $(ii)$, $(iii)$ and $(iv)$
that was violated in $G$ is still violated in $G'$. Since all
cutpoints of $H'$ are labeled $A$, $G'$ has one less cutpoint
labeled $B$ than $G$. By inductive hypothesis, $G'$ contains a
graph $F$ of $\F$ as induced subgraph. Notice that $G' -
\{v_1,v_1'\}$ and $G' - \{v_2,v_2'\}$ are isomorphic to an induced
subgraph of $G$. So, since $F$ is connected, if $F$ does not
contain one of $\{v_1,v_2\}$, then $G$ contains $F$ as an induced
subgraph. If $F$ contains $v_1$ and $v_2$, by Proposition
\ref{prop:F}, $F$ contains $H' \cup \{v_1',v_2'\}$, and $H'$ is an
almost endblock of $F$. Let $F'$ be the graph obtained from $F$ by
applying the procedure given in the definition of $\F$ to the
vertices $v_1$ and $v_2$. Then $F'$ belongs to $\F$ and $F'$ is an
induced subgraph of $G$.

If $H^v$ has two cutpoints, conditions $(ii)$, $(iii)$ and $(iv)$
are ``locally'' satisfied by $v$ and $H^v$, and the label of the
other cutpoint of $H^v$ does not depend on the block $H^v$. We can
delete from $G$ all the connected components of $G - v$, except
the one containing $H^v-v$, and call $G'$ the obtained graph. The
block $H$ is now an endblock of $G'$, the label of every cutpoint
placed before $v$ in $s$ remains unchanged in a labeling of
$bc(G')$, and the condition among $(ii)$, $(iii)$ and $(iv)$ that
was violated in $G$ is still violated in $G'$. Moreover, $v$ is no
longer a cutpoint in $G'$, so $G'$ has one less cutpoint labeled
$B$ than $G$. By inductive hypothesis, $G'$ contains a graph $F$
of $\F$ as induced subgraph. Since $G'$ is an induced subgraph of
$G$, so is $F$. $\diamondsuit$

\medskip

As a block $H$ is $H^c$ for all but at most one of its cutpoints
$c$, item $(iii)$ of the previous claim implies that no block has
four cutpoints such that two of them labeled $B$, and item $(iv)$
of the previous claim implies that no block has three cutpoints
labeled $B$.

Since, by item $(i)$, no block has five or more cutpoints, the
possible label multisets for the blocks of $G$ are $\{A\}$,
$\{B\}$, $\{A,A\}$, $\{A,B\}$, $\{B,B\}$, $\{A,A,A\}$,
$\{A,A,B\}$, $\{A,B,B\}$, $\{A,A,A,A\}$ and $\{A,A,A,B\}$.

Let $i=1$, so $G_i$ has only one block $H_1$. Note that $H_1$ is
$H^c$ for every cutpoint $c$ of $G$ belonging to $H_1$. So,
considering the label multiset of the vertices of $H_1$, the cases
$\{A,A,A,B\}$ and $\{A,B,B\}$ cannot arise (by items $(iii)$ and
$(iv)$ of the claim, respectively). In the cases $\{A\}$, $\{B\}$,
and $\{A,A\}$, the block can be represented either as a line
clique or as a cross clique, satisfying the conditions. In the
cases $\{A,B\}$ and $\{B,B\}$, the block can be represented as a
cross clique where one of the labeled vertices is the farthest
North and South line, and the other one is the farthest East and
West line. In the cases $\{A,A,A\}$ and $\{A,A,B\}$, the block can
be represented as a cross clique where one of the labeled vertices
(the vertex labeled $B$ in the second case) is the farthest North
and South line, and the other two are the farthest East,
respectively West, line. In the case $\{A,A,A,A\}$, the block can
be represented as a cross clique where each labeled vertex
corresponds to the farthest North, South, East or West line.

We will proceed now by induction. Let $i > 1$, and let $v:=c(i)$, the
only cutpoint of $H_i$ appearing in $s$ before $H_i$. Let $H_j,
H_{j+1}, \dots, H_i$ be the blocks between $v$ and $H_i$
containing $v$ (it can be $j=i$). As noticed above,
 $H_j, H_{j+1}, \dots, H_i$ are endblocks, and since the first element of $s$ is a block, $j >
 1$. In particular, $H^{v} \subseteq G_{j-1}$. Notice also that for
 $j \leq k \leq i$ and for every cutpoint $c$ of $G$, different from $v$, that belongs to
 $H_k$, it holds $H_k = H^c$.

We know by inductive hypothesis that there is a \VO \
representation of $G_{j-1}$ such that each cutpoint $c$ of $G$
that belongs to $G_{j-1}$ corresponds to the farthest North,
South, East or West line of the line or cross representation of
the clique $H^c$ and, moreover, if $c$ is labeled $B$, then it
corresponds to the farthest North and South, or East and West
(simultaneously) line of the line or cross representation of the
clique $H^c$.

We will show that, possibly refining the grid, we can extend this
representation to a representation of $G_i$ with the desired
properties.

We will consider the possible cases for the label of $v$ and the
remaining labeled vertices of $H_j, \dots, H_i$.\\

\medskip

\underline{Case 1:} $v$ is labeled $B$. \\

Without loss of generality, assume that vertex $v$ corresponds to
the farthest North and South line of the representation of $H^v$,
say $P_v$. As $H^v$ is the only clique of $G_{j-1}$ containing
$v$, $P_v$ has two segments $P_v^N$ and $P_v^S$ that do not
intersect any other path in $G_{j-1}$, and each of them contains
an endpoint of $P_v$.

Since $v$ is labeled $B$, we have that the possible multisets for
the blocks $H_j, \dots, H_i$ are $\{B\}$, $\{A,B\}$, $\{B,B\}$,
$\{A,A,B\}$, $\{A,B,B\}$, and $\{A,A,A,B\}$. By the item $(ii)$ of
the claim, at most two of them have labels $\{A,A,A,B\}$ or
$\{A,B,B\}$ (there are exactly two such blocks in $G$, but some of
them may have index greater than $i$). We will assign to each
block a segment of $P_v^N$ or $P_v^S$, in such a way that the
blocks having labels $\{A,A,A,B\}$ or $\{A,B,B\}$ receive the
segments of $P_v^N$, respectively $P_v^S$, that contain an
endpoint of $P_v$. It is easy to see that we can extend the
representation to a  \VO \ representation of $H$ satisfying the
required properties: in the case of labels $\{B\}$, we add the
remaining vertices in a line clique on the assigned segment; in
the case of labels $\{A,B\}$ or $\{B,B\}$, we add the remaining
vertices in a cross clique on the assigned segment, in such a way
that the other labeled vertex corresponds to the farthest East and
West line of the clique; in the case of labels $\{A,A,B\}$, we add
the remaining vertices in a cross clique on the assigned segment,
in such a way that the other two labeled vertices correspond to
the farthest East, respectively West, line of the clique; in the
case of labels $\{A,B,B\}$, we add the remaining vertices in a
cross clique on the assigned segment, in such a way that the other
vertex labeled $B$ corresponds to the farthest East and West line
of the clique, and the third labeled vertex corresponds to the
farthest North line if the segment assigned contains the North
endpoint of $P_v$, and to the farthest South line, otherwise;
finally, in the case of labels $\{A,A,A,B\}$, we add the remaining
vertices in a cross clique on the assigned segment, in such a way
that two of the other labeled vertices correspond to the farthest
East, respectively West, line of the clique, and the third labeled
vertex corresponds to the farthest North line if the segment
assigned contains the North endpoint of $P_v$, and to the farthest
South line, otherwise.

For a scheme, see the leftmost draw in Figure \ref{f:scheme}.\\

\begin{figure}[h]
\centering{
\includegraphics[width=.7\textwidth]{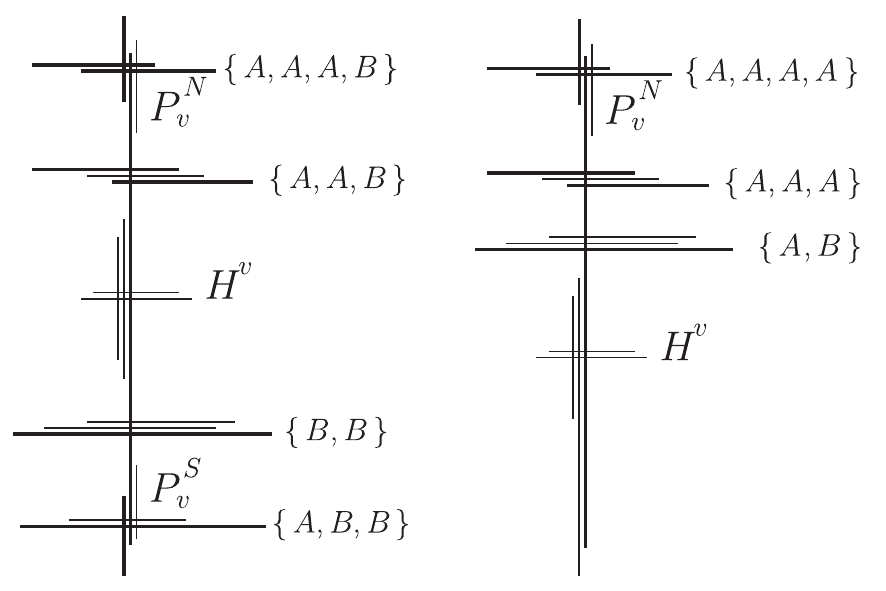}
\caption{Scheme for the extension of a representation of $G_{j-1}$
to $G_{i}$. The cutpoints are represented by bold
lines.}\label{f:scheme}}
\end{figure}

\medskip

\underline{Case 2:} $v$ is labeled $A$. \\

Without loss of generality, assume that vertex $v$ corresponds to
the farthest North line of the representation of $H^v$, say $P_v$.
As $H^v$ is the only clique of $G_{j-1}$ containing $v$, $P_v$ has
a segment $P_v^N$ that does not intersect any other path, and
contains the North endpoint of $P_v$.

Since $v$ is labeled $A$, the possible multisets for the blocks $H_j, \dots, H_i$ are
$\{A\}$, $\{A,A\}$, $\{A,B\}$, $\{A,A,A\}$, $\{A,A,B\}$,
$\{A,A,A,A\}$. Notice that, since for
 $j \leq k \leq i$ and for every cutpoint $c$ of $G$, different from $v$, that belongs to
 $H_k$, it holds $H_k = H^c$, the multisets $\{A,A,A,B\}$ and $\{A,B,B\}$
cannot arise (by items $(iii)$ and $(iv)$ of the claim,
respectively).

By the labeling rules, at most one block in $H_j, \dots, H_i$ has
labels $\{A,A,A,A\}$ or $\{A,A,B\}$. We will assign to each block
a segment of $P_v^N$, in such a way that the block having labels
$\{A,A,A,A\}$ or $\{A,A,B\}$ receives the segment of $P_v^N$ that
contains the North endpoint of $P_v$. It is easy to see that we
can extend the representation to a  \VO \ representation of $H$
satisfying the required properties: in the case of labels $\{A\}$,
we add the remaining vertices in a line clique on the assigned
segment; in the case of labels $\{A,A\}$ or $\{A,B\}$, we add the
remaining vertices in a cross clique on the assigned segment, in
such a way that the other labeled vertex corresponds to the
farthest East and West line of the clique; in the case of labels
$\{A,A,A\}$, we add the remaining vertices in a cross clique on
the assigned segment, in such a way that the other two labeled
vertices correspond to the farthest East, respectively West, line
of the clique; in the case of labels $\{A,A,B\}$, we add the
remaining vertices in a cross clique on the assigned segment, in
such a way that the vertex labeled $B$ corresponds to the farthest
East and West line of the clique, and the third labeled vertex
corresponds to the farthest North line of the clique; finally, in
the case of labels $\{A,A,A,B\}$, we add the remaining vertices in
a cross clique on the assigned segment, in such a way that two of
the other labeled vertices correspond to the farthest East,
respectively West, line of the clique, and the third labeled
vertex corresponds to the farthest North line of the clique.

For a scheme, see the rightmost draw in Figure \ref{f:scheme}.\\

The minimality holds by the equivalence of \VO  \ and $\F$-free
within block graphs, and Corollary \ref{cor:minimal}.
\end{proof}

\begin{cor} Let $G$ be a  chordal diamond-free $VPG$ graph. $G$ is
a $B_0$-VPG graph if and only if  $G$ is $\F$-free.\end{cor}

\begin{proof} It follows directly by the fact that block graphs are connected chordal diamond-free graphs.
\end{proof}

\section{Conclusion}

In this paper we considered $B_0$-VPG graphs, that is,
intersection graphs of paths on a grid such that each path is
either a horizontal path or a vertical path on the grid. We
characterized whether a block graph is a $B_0$-VPG graph in terms
of minimal forbidden induced subgraphs.

The proof of Theorem \ref{t:main} (i.e., the labeling process and
the conditions of the claim) provides an alternative recognition
and representation algorithm for $B_0$-VPG graphs in the class of
block graphs. This algorithm is a certifying algorithm, since if the answer is negative it provides a minimal forbidden induced subgraph.

\end{document}